\documentclass[12pt]{article}
\usepackage{mystyle}
\usepackage{amsfonts, amsmath, amssymb, amsthm}
\usepackage[T1]{fontenc}
\usepackage{amsmath,amsthm,amssymb,amsfonts}
\usepackage{hyperref}
\usepackage{breqn}
\usepackage{epsfig,color,graphicx,enumerate}

\usepackage{accents}
\usepackage{tabularx}
\numberwithin{equation}{section}
\newtheorem{teo}{Theorem}
\newtheorem{lemma}[teo]{Lemma}

\newtheorem{prop}[teo]{Proposition}
\newtheorem{deff}[teo]{Definition}
\theoremstyle{remark}
\newtheorem{rmk}[teo]{Remark}

\usepackage{wrapfig}
\usepackage{tcolorbox} 
\usepackage{times}
\usepackage{cite}
\usepackage{color,calc}

\newcommand{\mip}[1]{\textcolor{green!40!black}{#1}}

\usepackage{color}
\usepackage{dsfont}
\usepackage[T1]{fontenc}
\usepackage{authblk}
\usepackage{graphicx,tikz,xcolor}
\usepackage{xcolor,eucal,enumerate,mathrsfs}
\usepackage{epsfig} 
\usepackage{cancel}
\usepackage{setspace}

\usepackage{mystyle}

\def\I{\mathcal{I}}

\def\N{\mathbb N}
\def\R{\mathbb R}
\def\Z{\mathbb Z}

\def\P{\mathcal{P}}

\newcommand{\prob}{\mathcal{P}}

\def\ep{\epsilon}

\def\Vee{{V}_{ee}}

\def\WN{\mathcal{W}^{N}}

\def\Te{T}
\def\Vne{{V}_{{\rm ne}}}
\def\vne{v_{{\rm ne}}}
\def\Sym{\mathfrak{S}}

\def\VSCE {\operatorname{V}_{ee}^{\operatorname{SCE}}}

\def\Vee{V_{ee}}
\def\Te{T}
\def\Vne{V_{ne}}

\def\ep{\varepsilon}

\def\a{\alpha}
\def\b{\beta}
\def\rae{r^{(\eta)}_\a}
\def\rbe{r^{(\eta)}_\b}
\def\na{{N_\alpha}}
\def\nb{{N_\beta}}
\def\ra{\rho_\alpha}
\def\rb{\rho_\beta}

\def\re{\rho_\eta}

\def\HK{\mathsf{HK}}
\def\Vee{\operatorname{V}_{ee}}

\newcommand{\OT}[2]{\mathsf{OT}_{#1}[#2]}
\newcommand{\EOT}[2]{{\mathsf{EOT}^{\epsilon}_{#1}[#2]}}
\newcommand{\Dep}[2]{{\mathsf{D}^{\epsilon}_{#1}[#2]}}

\newcommand{\bee}[1]{\begin{equation}\label{#1}}
\newcommand{\eee}{\end{equation}}

\def\bal{\begin{aligned}}
\def\eal{\end{aligned}}

\def\XXint#1#2#3{{\setbox0=\hbox{$#1{#2#3}{\int}$} 
  \vcenter{\hbox{$#2#3$}}\kern-.5\wd0}}

\usepackage{color}
\usepackage{pstricks}
\usepackage{ifthen}
\newrgbcolor{mirceacolor}{.1 .4 .1}
\newrgbcolor{mirceacommentcolor}{.8 0 0}
\newrgbcolor{augustocolor}{.1 .1 .7}
\newrgbcolor{adolfocolor}{.1 .7 .1}
\usepackage{csquotes}

\newboolean{DEBUG}
\setboolean{DEBUG}{true}

{}

\ifthenelse {\boolean{DEBUG}}
{}

\ifthenelse {\boolean{DEBUG}}
{}

\title{Existence and uniqueness of Monge minimizers for a Multi-marginal Optimal Transport problem with intermolecular interactions cost}

\author[1,2,3]{Augusto Gerolin}
\author[4,5]{Mircea Petrache}
\author[1]{Adolfo Vargas-Jim\'enez}
\affil[1]{Department of Mathematics and Statistics, University of Ottawa}
\affil[2]{Department of Chemistry and Biomolecular Sciences, University of Ottawa}
\affil[3]{Nexus for Quantum Technologies, University of Ottawa}
\affil[4]{Department of Mathematics, Pontificia Universidad Cat\'olica de Chile}
\affil[5]{Institute for Mathematical and Computational Engineering, Pontificia Universidad Cat\'olica de Chile}

\begin{document}
\maketitle

\begin{abstract}
\noindent
We investigate a new multi-marginal optimal transport problem arising from a dissociation model in the Strong Interaction Limit of Density Functional Theory. In this short note, we introduce such  dissociation model, the corresponding optimal transport problem as well as show preliminary results on the existence and uniqueness of Monge solutions assuming absolute continuity of at least two of the marginals. Finally, we show that such marginal regularity conditions are necessary for the existence of an unique Monge solution.
\end{abstract}

\section{Introduction}

In this work, we introduce the following variational problem, motivated by Density Functional Theory:
\begin{equation}\label{maincost}
\inf_{\gamma\in\Pi_{\na}(\ra)\otimes\Pi_{\nb}(\rb)}\int_{\R^{d(\na+\nb)}}\sum^{\na}_{i=1}\sum^{\nb}_{j=1}x_i\cdot y_j^{*}d\gamma(x_1,\dots,x_\na,y_{1},\dots,y_\nb),
\end{equation}
where $y_j^{*}=(-2y_j^{1}, y_j^{2}, \ldots, y_j^{d})$ for each $j\in \lbrace 1,\dots,\nb\rbrace$,
$\na,\nb\in\N$ are integers, $\ra$ and $\rb$ are probability measures, and $\Pi_{\na}(\ra)\otimes\Pi_{\nb}(\rb)$ denotes the set of transport plans $\gamma=\gamma_{\alpha}\otimes\gamma_{\beta}$ such that $\gamma_{\alpha}$ has $\na$ marginals equal to $\ra$ and $\gamma_{\beta}$ has $\nb$ marginals equal to $\rb$. 

The multi-marginal optimal transport problem \eqref{maincost} thus has an attractive harmonic cost in the direction $x_1\in \R$ and repulsive harmonic cost in all other space directions $(x_2,\dots,x_d)\in \R^{d-1}$. The problem with fully attractive or repulsive harmonic cost have been considered, e.g. in \cite{GanSwi-CPAM-98,GerKauRaj-SIAMJMA-19,DMaGerNen-TOOAS-17,GerThesis}.

The analysis of \eqref{maincost} naturally appears in the study of the dissociation energy in a Density Functional Theory framework (DFT) \cite{KooGor-JPCL-19,FriGerGor-arxiv-22}. The dissociation energy is the energy required to break apart a chemical bond and translate the constituent atoms or molecules to an infinitely large distance from each other.

Our mathematical analysis focuses on the asymptotics of the problem and minimizers, in the limit when the ground-state energy of a cluster of interacting molecules $\alpha\beta$ becomes infinitely separated into individual clusters $\alpha$ and $\beta$. In our model, the probabilities $\ra$ and $\rb$ in \eqref{maincost} are the single-particle densities of the isolated systems $\alpha$ and $\beta$. Such dissociation limit is of significant importance in the study of molecular properties and chemical reactions, such as energy transfer, bond strength, stability of compounds, atomic and molecular spectroscopy.\vspace{2mm}

\noindent
\textbf{Multi-marginal Optimal Transport theory:}
In the multi-marginal optimal transport, we correlate a finite number of probability measures to minimize some notion of overall cost. Formally, for given Borel probability measures $\rho_1, \ldots, \rho_N$  on open sets $X_i \subseteq \R^{d}$, $i=1, \ldots, N$ (respectively), and $c$ a given cost function on the product space $X:=\prod_{i=1}^{N}X_i$, one seeks to minimize the \textit{total cost}
 \begin{equation}\label{KP}
\displaystyle \int_{X} c(x_{1},\ldots, x_{N})d\gamma,\tag{KP}
 \end{equation}
 among all Borel probability measures $\gamma$ on $X$  whose marginals are the $\rho_i$. We say that $\rho_i$ is the $i$th marginal of $\gamma$ if, for any Borel set $A\subseteq X_i$, we have $\gamma(X_{1}\times\ldots \times X_{i-1}\times A \times X_{i+1}\times \ldots \times X_{N})=\rho_{i}(A)$. This formulation is known as the Kantorovich Problem (KP) and it is, in fact, a relaxation of a restricted problem: the Monge Problem (MP).  
 
 In the Monge formulation, one seeks to minimize
 \begin{equation}\label{MP}
 \displaystyle \int_{X_{1}} c(x_{1},T_{2}x_{1},\ldots, T_{N}x_{1})d\rho_{1},\tag{MP}
 \end{equation}
 among all $(N-1)$-tuples of maps $(T_{2}, \ldots, T_{N})$ such that  $(T_{i})_{\sharp} \rho_{1}=\rho_{i}$ for all $i=2, \ldots, N$, where $(T_{i})_{\sharp} \rho_{1}$ denotes the \textit{image measure} of $\rho_{1}$ through $T_{i}$, defined by $(T_{i})_{\sharp} \rho_{1}(A) = \rho_{1}(T_{i}^{-1}(A))$, for any Borel set $A\subseteq X_{i}$.
 \par
 When $N=2$, (KP) and (MP) reduce respectively  to the  Kantorovich and Monge formulations of the classical optimal transport problem. This case is well understood; in particular, if the map $x_{2}\mapsto D_{x_{1}}c(x_{1},x_{2})$ is injective for each fixed $x_1$ and $\rho_1$ is absolutely continuous with respect to the $d$-dimensional Lebesgue measure $\mathcal{L}^{d}$, there exists a unique solution to (KP) and it is induced by a measurable map \cite{Brenier89, McCann01, San-Book-15,Vil-BOOK-03}. 
 
Under the same regularity condition on the first marginal $\rho_1$, Kim and Pass \cite{KimPass2014} extended this result to the multi-marginal case covering a wide class of cost functions \cite{Carlier2010, GanSwi-CPAM-98, Hei-CRMASP-02, KimPass2015, Pass2015, Pas-DCDS-14}, namely the cost function satisfies the so-called  \textit{twist condition on $c$-splitting sets} \cite{fathi2010optimal,champion2014twist}, i.e.  the mapping
$$(x_{2}, \ldots, x_{N})\mapsto D_{x_{1}}c(x_{1}^{0},x_2,\ldots, x_{N}) $$ is injective on
the subset of $S$ where $D_{x_1}c(x_1^{0}, x_2, \ldots, x_N)$ exists, for each fixed $x_1^{0}\in X_1$ and $c$-splitting set $S\subseteq \{x_1^{0}\}\times X_2\times \ldots X_N$. See Definition  \ref{splittingsets}.

Although being fairly general, the \textit{twist condition on $c$-splitting sets} does not hold for many costs (e.g., cyclic Euler cost \cite{Brenier89}), or is very difficulty to be verified (e.g., Coulomb cost), see \cite{PassVar2021,PassVar2022,PassVar2023} for details and further examples. When $N_{\alpha}>1$,  the cost in \eqref{maincost} is not twisted on splitting sets and an alternative approach must be used. In that case, Pass$\&$Vargaz-Jim\'enez \cite{PassVar2023} show that extra regularity conditions on some of the marginals (in addition to a regularity condition on $\rho_1$) are sufficient to guarantee uniqueness of Monge solution.\vspace{2mm}



\noindent
\textbf{Main result and proof strategy:} In this paper, we show that, under some regularity conditions on some of the marginals, the multi-marginal optimal transport problem \eqref{maincost} admits a unique Monge solution, see Theorem \ref{Eqn:10}. The general case $\na>1$ requires the use of a more general approach developed by Pass and Vargas-Jim\'enez \cite{PassVar2022}, which generalizes the twist on $c$-splitting sets condition. Roughly speaking, Pass$\&$Vargaz-Jim\'enez condition requires the mapping $(x_2, \ldots, x_N)\mapsto D_{x_1}c(x_1, x_2, \ldots, x_N)$ to be injective on special subsets generated by $c$-splitting functions, see Definition \ref{Art3:10} for details. In particular, this guarantees that every solution to (KP) is concentrated on a graph of a measurable map, and therefore, by simply using a standard argument uniqueness of Monge solution for \eqref{maincost} is obtained. Finally, Lemma \ref{Eqn:15} shows that such regularity conditions on the marginals are necessary for the uniqueness of Monge solutions.

\vspace{2mm}


\noindent
\textbf{Organization of the paper:} In section \ref{sec:dis}, we define the dissociation energy (as defined in \eqref{eq:dissen})  of many-electron quantum systems as well as the corresponding energy in the so-called Strong Interaction limit of Density Functional Theory. In section \ref{sec:dispersionSCE}, we introduce the mathematical framework and compute the asymptotic development of the multi-marginal optimal transport problem with Coulomb cost in the dissociation limit. Finally, in section \ref{sec:MOT} we prove preliminary results on the existence and uniqueness of Monge solutions for the  problem \eqref{maincost}.\vspace{2mm}
\noindent


\section{Dissociation energy of many-electrons quantum systems}\label{sec:dis}

We consider a quantum mechanical system of $N=\na+\nb$ non-relativistic electrons (of mass $m_e$ and charge $-e$), representing clusters of molecules $\alpha$ and $\beta$ respectively, interacting with each other, moving around classical nuclei with positions $\mathbf{R}_1,\dots,\mathbf{R}_M \in \R^d$ and charges $Z_1e,\dots,Z_Me$ (Born-Oppenheimer approximation). 

The electrons are described by a wave function $\Psi:(\R^d\times \Z_2)^N\to\mathbb{C}$ of $N$ positions $z \in \R^d$ and spin coordinates $s_i\in\lbrace\uparrow,\downarrow\rbrace=\Z_2$ that are antisymmetric with respect to permutations of the electron coordinates,
\begin{equation} \label{eq:anti}
\Psi(z_{\sigma(1)},s_{\sigma(1)},\dots,z_{\sigma(N)},s_{\sigma(N)}) =
\operatorname{sign}(\sigma)\Psi(z_1,s_1,\dots,z_N,s_N), \quad \sigma
\in \Sym_N, 
\end{equation}
where $\Sym_N$ denotes the group of permutations of the indices $1,...,N$. The set of square-integrable $N$-electron wave functions, $\{\Psi\in L^2((\R^d\times\Z_2)^N;\mathbb{C}) \, : \, \eqref{eq:anti} \text{ holds }\}$, will be denoted $\bigwedge^N_{i=1}L^2(\R^d\times\Z_2;\mathbb{C})$.

The total energy $E^{\alpha\beta}[\Psi,v]$ of the fermonic state of the coupled system $\alpha\beta$ with external potential $v:\R^d\to\R$ is given, in atomic units, by
\begin{equation}\label{intro:energyE}
\quad E^{\alpha\beta}[\Psi,v] =  \Te[\Psi]+ \Vee[\Psi] + \Vne[\Psi,v],
\end{equation}
where $\Te[\Psi]$ is the \textit{kinetic energy}
\[
\Te[\Psi] = \dfrac{1}{2}\sum_{s_1 \in \Z_2} \int_{\R^3}\dots\sum_{s_{\na} \in \Z_2} \int_{\R^3}\sum^{N}_{i=1}\vert
\nabla\Psi(z_1,s_1\dots,z_N,s_N)\vert^2dz_1\dots dz_N,
\]
$\Vee[\Psi]$ is the Coulomb electronic-electronic interaction energy 
\[
\Vee[\Psi] = \sum_{s_1 \in \Z_2} \int_{\R^d}\dots\sum_{s_N \in \Z_2}
\int_{\R^d} \sum_{1\le i<j\le N}\dfrac{1}{\vert z_i-z_j\vert}\vert\Psi(z_1,s_1\dots,z_N,s_N)\vert^2dz_1\dots dz_N,
\]
and $\Vne[\Psi,v]$ is the electron-nuclei interaction energy
\[
\Vne[\Psi,v] = \sum_{s_1 \in \Z_2} \int_{\R^d}\dots\sum_{s_N \in \Z_2} \int_{\R^d}\sum^N_{i=1}v(z_i)\vert
\Psi(z_1,s_1\dots,z_N,s_N)\vert^2dz_1\dots dz_N.
\]
Typically, $v$ is the Coulomb potential generated by $M$ nuclei which are at positions $R_\nu$ and have charges $Z_\nu$,
\begin{equation}\label{eqn:typicalv}
v(z) = - \sum_{\nu=1}^M \frac{Z_\nu}{\vert z - R_\nu\vert}.
\end{equation}

The ground state energy is defined by
\begin{equation}\label{eq:GSEalphabeta}
E_0[v] = \min \left\lbrace E^{\alpha\beta}[\Psi,v] \, : \, \Psi\in L^2((\R^d\times\Z_2)^N;\mathbb{C}), \Vert \Psi\Vert_{2}= 1 \text{ and } \eqref{eq:anti} \text{ holds}  \right\rbrace.
\end{equation}

Similarly, we define the total energies $E^{\alpha},E^{\beta}$ of the isolated clusters of molecules $\alpha$ and $\beta$ 
\begin{equation}\label{eq:Energyalpha}
\quad E^{\alpha}[\Psi^{\alpha},v^{\alpha}] =  \Te^{\alpha}[\Psi^{\alpha}]+ \Vee[\Psi^{\alpha}] + \Vne[\Psi^{\alpha},v^{\alpha}], \text{ and }
\end{equation}
\begin{equation}\label{eq:Energyalpha}
\quad E^{\beta}[\Psi^{\beta},v^{\beta}] =  \Te^{\beta}[\Psi^{\beta}]+ \Vee[\Psi^{\beta}] + \Vne[\Psi^{\beta},v^{\beta}], 
\end{equation}
as well as their corresponding ground-state energies $E_0^{\alpha}$ and $E_0^{\beta}$
\begin{equation}\label{eq:GSEalpha}
E^{\alpha}_0[v^{\alpha}] = \min \left\lbrace E^{\alpha}[\Psi^{\alpha},v^{\alpha}] \, : \, \Psi\in L^2((\R^d\times\Z_2)^{N_\alpha};\C), \Vert \Psi^{\alpha}\Vert_{2}= 1 \text{ and } \eqref{eq:anti} \text{ holds}  \right\rbrace, \text{ and }
\end{equation}
\begin{equation}\label{eq:GSEbeta}
E^{\beta}_0[v^{\beta}] = \min \left\lbrace E^{\beta}[\Psi^{\beta},v^{\beta}] \, : \, \Psi\in L^2((\R^d\times\Z_2)^{N_\beta};\C), \Vert \Psi^{\beta}\Vert_{2}= 1 \text{ and } \eqref{eq:anti} \text{ holds}  \right\rbrace.
\end{equation}

We are interested in the dissociation energy of the system $\alpha\beta$, i.e. the energy difference between the fully coupled system $\alpha\beta$ and the sum of the energies of the isolated systems $\alpha$ and $\beta$:
\begin{equation}\label{eq:dissen}
E_{\text{diss}} = E_0 - (E_0^{\alpha} + E_0^{\beta}).
\end{equation}

In the following, we will focus on the so-called \textit{Strong-Interaction limit of Density Functional Theory}, which was introduced in \cite{KooGor-JPCL-19}.

\noindent
\subsubsection*{The Strong Interaction limit of Density Functional Theory (DFT)} By integrating the $N$-point probability distribution over the spins, we obtain the $N$-point position density,  
\begin{equation}\label{eq:Nppd}
\pi_N^{\Psi}(z_1,\dots,z_N) := \sum_{s_1,\dots,s_N\in\Z_2} \vert \Psi(z_1,s_1,\dots,z_N,s_N)\vert^2, \quad \Psi \in \WN.
\end{equation}

\quad The single particle density $\rho_{\Psi}(z_j)$ is then obtained by integrating out all but one electron position $z_j\in\R^d$, 
\begin{equation}\label{eq:rho1}
\rho_{\Psi}(z_j) := \textcolor{black}{N}\int_{\R^{d(N-1)}}  \pi^{\Psi}_N(z_1,z_2,\dots, z_j,\dots,z_N) \prod_{i\neq j} dz_i, \quad ~\forall~ j \in \{1,...,N\}.
\end{equation}
We denote by $\Psi \mapsto \rho$ the relation between $\Psi$ and $\rho$ given by equations  \eqref{eq:Nppd}, \eqref{eq:rho1}. This means that the wave function $\Psi$ has single-electron density $\rho$.

\quad Following the work of Hohenberg and Kohn \cite{HohKoh-PR-64}, Levy \cite{Lev-PRA-82} and Lieb \cite{Lie-IJQC-83} showed that the electronic ground state problem \eqref{eq:GSEalphabeta} can be recast as a minimization over single-electron densities $\rho$ instead of many-electron wavefunctions $\Psi$: 
\begin{equation}\label{eq.groundstatedensity}
E_0[\vne^{\alpha\beta}] = \inf_{\rho \in \mathcal{D}^N} \bigg\lbrace F^{\alpha\beta}_{\rm LL}[\rho] +
N\int_{\R^d}\vne^{\alpha\beta}(r)\rho(r)dr \bigg\rbrace,
\end{equation}
with
\begin{equation}\label{eq.FHK}
F^{\alpha\beta}_{\rm LL}[\rho] = \min\bigg\lbrace \Te[\Psi] + \Vee[\Psi] : \Psi \in \WN, \Psi \mapsto \rho \bigg\rbrace, \medskip
\end{equation}
where $F_{\rm LL}[\rho]$ is the Levy-Lieb functional. The space $\mathcal{D}^N$ (characterized in \cite{Lie-IJQC-83}) is defined as the set of densities $\rho$ such that $\Psi\mapsto\rho$ for some $\Psi \in \WN$, 
\[
\WN = \left\lbrace \Psi \in \bigwedge_{i=1}^N H^1(\R^d\times\Z_2;\mathbb{C}) \, : \, \sum_{s_1,...,s_N\in\Z_2}\int_{\R^{dN}} \vert \nabla \Psi\vert^2 dz_1\dots dz_N  < +\infty, \; ||\Psi|| = 1  \right\rbrace.
\]
Analogously, when we replace $\rho, N$ by $\rho_\alpha, N_\alpha$ or by $\rho_\beta,N_\beta$, we obtain the ground state energies of the cluster of molecules $\alpha$ and $\beta$, which will be denoted, respectively, by

\begin{equation}\label{eq.groundstatedensity}
E^{\alpha}_0[\vne^{\alpha}] = \inf_{\ra \in \mathcal{D}^{\na}} \bigg\lbrace F^{\alpha}_{\rm LL}[\ra] +
\na\int_{\R^d}\vne^{\alpha}(r)\ra(r)dr \bigg\rbrace, \quad \text{ and }
\end{equation}
\begin{equation}\label{eq.groundstatedensity}
E^{\beta}_0[\vne^{\beta}] = \inf_{\rb \in \mathcal{D}^{\nb}} \bigg\lbrace F^{\beta}_{\rm LL}[\rb] +
\nb\int_{\R^d}\vne^{\beta}(r)\rb(r)dr \bigg\rbrace.\vspace{2mm}
\end{equation}

\noindent
\textbf{Strong Interaction limit of Density Functional Theory:} The Strong-Interaction limit functional is the limit when $\hbar\to 0^+$ of the Hohenberg-Kohn-Levy-Lieb functional \eqref{eq.FHK} \cite{SeiGorSav-PRA-07,CotFriKlu-CPAM-13,CotFriKlu-ARMA-18} (see also \cite{FriGerGor-arxiv-22,VucGerDaaBahFriGor-Wire-23} for a complete overview), and has the following form:
\begin{align}\label{eq.SIL}
    \VSCE[\rho] &:=\min_{\gamma\in \Pi(\rho)}\int_{\R^{dN}} \Vee(z_1,...,z_N) \, d\gamma(z_1,...,z_N),\\
    \text{where}\quad &\Vee(z_1,\dots,z_N):=\sum_{1\le i<j\le N}\dfrac{1}{\vert z_i-z_j\vert}\nonumber\\
    \text{and}\quad &\Pi(\rho):=\left\{\gamma\in{\cal P}_{sym}(\R^{dN}):\ \gamma\mapsto\rho\right\},\nonumber 
\end{align}
in which ${\cal P}_{sym}(\R^{dN})$ denotes the set of probability measures that are invariant under permutation of the $N$ coordinates in $\R^{dN}=(\R^d)^N$ and, analogously to \eqref{eq:rho1}, the notation $\gamma\mapsto\rho$ means that $\gamma$ has all marginals equal to $\rho$, i.e. $(e_i)_{\sharp}\gamma = \rho, \forall i \in \lbrace 1,\dots,N\rbrace$ with $e_i:\R^{dN}\to \R^d$ the projection operator. The equation \eqref{eq.SIL} corresponds to an optimal transport problem with finitely many marginals and Coulomb cost \cite{ButDepGor-PRA-12,Dep-ESAIMMMNA-15,CotFriKlu-CPAM-13,CotFriKlu-ARMA-18,DMaGerNen-TOOAS-17,GerKauRaj-ESAIMCOCV-19,GerKauRaj-CVPDE-20, bindini2017optimal, lewin2018semi}.

\subsection{A Dissociation model in the Strong Interaction limit of DFT}\label{sec:dispersionSCE}

Let $\eta>0$ be a positive number \mip{and} $\ra,\rb\in\mathcal{P}(\R^3)$ be two probability measures in $\R^3$. Define the probability measure $\rho_{\eta}$ by
\begin{align}\label{eq:rhoR}
    \rho_{\eta}(z)&:=\frac{\na}{\na+\nb}\ra (z-r_\a)+\frac{\nb}{\na+\nb}\rb(z-r_\b), \quad  \\  
&\text{ where } \, r_\a=\rae=\frac{1}{2\eta}\mathbf e_1 \quad \text{and} \quad r_\b=\rbe=-\frac{1}{2\eta}\mathbf e_1. \nonumber
\end{align}
 
The probability density $\rho_{\eta}$ is the single particle density of the composite system $\alpha\beta$ of $N=\na+\nb$ electrons formed by the isolated systems $\alpha$ and $\beta$ having single-particle densities given by, respectively, $\ra$ and $\rb$. This model captures the change of the single-particle density $\rho_{\eta}$ of the system $\alpha\beta$ on varying the distance $R=\eta^{-1}$ between the two clusters $\alpha$ and $\beta$.\vspace{2mm}

We are interested in the asymptotic development in $\eta\to 0^+$ of the SCE functional \eqref{eq.SIL} $\VSCE[\re]$.\vspace{2mm}

\noindent
{\bf Center-of-molecule coordinates.} As indicated by \eqref{eq:rhoR}, as $|\rae-\rbe|\to\infty$ the positions of the two molecules $\a,\b$ change, but the particle densities $\ra$ and $\rb$ do not change. It is then natural to work with new coordinates, denoted $x,y$, defined via
\begin{equation}\label{coordinates}
 x:=z-\rae, \qquad y:=z-\rbe.
\end{equation}
Then \eqref{eq:rhoR} can be written as
\begin{equation}\label{rhoR1}
 \rho_\eta(z):=\frac{\na}{\na+\nb}\ra(x)+\frac{\nb}{\na+\nb}\rb(y).
\end{equation}
We are now in the position to study the dissociation or electron-electron interaction of the entire cluster. As molecule $\a$ has $\na$ electrons and molecule $\b$ has $\nb$ electrons, it is natural to use center-of-molecule coordinates in which $\na$ of the positions $z_1,\dots, z_N$ are re-centered at $r_\alpha^{(\eta)}$ and the remaining $\nb$ are re-centered at $r_\beta^{(\eta)}$. With these choices, we rewrite the Coulomb electronic-electronic potential $\Vee$ as
\begin{align}\label{veta} 
    \Vee^\eta(\vec x,\vec y)&=\sum^{\na}_{i=1}\sum^{\nb}_{j=1}\dfrac{\eta}{\sqrt{1-2\eta(x^1_i-y^1_j)+\eta^2|x_i-y_j|^2}} + \Vee^{\alpha}(\vec{x}) + \Vee^{\beta}(\vec{y}), \\
    &\text{ where } \Vee^{\alpha}(\vec{x}) =\sum_{1\leq i<j\leq \na}\dfrac{1}{\vert x_i-x_j\vert}, \quad \text{and} \quad \Vee^{\beta}(\vec{y}) = \sum_{1\leq i<j\leq \nb}\dfrac{1}{\vert y_i-y_j\vert}.\nonumber\vspace{2mm}
\end{align}

\noindent
{\bf Plans that respect the molecule structure.} In order to model the dissociation via a simplified energy we introduce the simplified minimization in which $\na$ electrons are assigned to molecule $\a$ and $\nb$ electrons are assigned to molecule $\b$. Thus rather than admissible plan sets $\Pi(\re)$ (from \eqref{eq:rhoR}) we are led to consider the following spaces:
\begin{eqnarray*}
    \lefteqn{\Pi_{\na,\nb}^\eta(\ra,\rb)}\\
    &:=&
    \left\{\gamma\in \mathcal P(\R^{dN}):\ \exists \, \gamma_\a\in \Pi_\na(\ra),\gamma_\b \in \Pi_\nb(\rb)\mbox{ such that }\gamma=\left((\tau_{\rae})_\sharp \gamma_\a \otimes (\tau_{\rbe})_\sharp \gamma_\b\right)_{sym}\right\}.
\end{eqnarray*}
Here for $x\in\mathbb R^d$ we denote by $\tau_x:\mathbb R^{Kd}\to\mathbb R^{Kd}$ the translation operation $(x_1,\dots,x_K)\mapsto(x_1+x,\dots,x_K+x)$ and for a measure $\widetilde\gamma$ over $\mathbb R^{Nd}$ we set $(\widetilde \gamma)_{sym}:=\frac{1}{N!}\sum_{\sigma\in\Sym_n}\sigma_\sharp\widetilde\gamma$, where a permutation $\sigma$ acts by $(x_1,\dots,x_N)\mapsto(x_{\sigma(1)},\dots,x_{\sigma(N)})$.

In other words, for defining $\Pi_{\na,\nb}^\eta(\ra,\rb)$ we take $\gamma_\a,\gamma_\b$ with respectively $\na$ marginals equal to $\ra$ and $\nb$ marginals equal to $\rb$ and then we apply the translations by $\rae,\rbe$ to them. Finally, we symmetrize the so-obtained plan, an operation that insures direct comparability to the set of competitors for $\Pi(\re)$ from \eqref{eq.SIL}. Indeed, we have the following:
\begin{lemma}\label{restrictmarginal} Let $\eta>0$, $\rae,\rbe\in\R^d$, $\ra,\rb\in\prob(\R^d)$ and $\re$ as defined in \eqref{rhoR1}. Then, $\Pi_{\na,\nb}^\eta(\ra,\rb)\subset \Pi(\re)$.
\end{lemma}

\begin{proof}
Let $\gamma\in \Pi^{\eta}(\ra,\rb)$, $\gamma_{\a} \in \Pi_{N_\a}(\a)$ and let $\gamma_{\b}\in \Pi_{N_\b}(\b)$ be such that $\gamma = \left((\tau_{\rae})_{\sharp}(\gamma_\a)\otimes(\tau_{\rbe})_{\sharp}(\gamma_\b)\right)_{sym}$. Clearly, $\gamma \in \P(\R^d)$ and for every $i\in \{1,\dots,N\}$, we have
\begin{align*}
(e_i)_{\sharp}\gamma &= \left(e_i\right)_{\sharp}\left((\tau_{\rae})_{\sharp}(\gamma_\a)\otimes(\tau_{\rbe})_{\sharp}(\gamma_\b)\right)_{sym}\\
&= \left(e_i\right)_{\sharp}\left(\frac{1}{N!}\sum_{\sigma\in\Sym_N}\sigma_{\sharp}\left((\tau_{\rae})_{\sharp}(\gamma_\a)\otimes(\tau_{\rbe})_{\sharp}(\gamma_\b)\right)\right)\\
&= \left(e_i\right)_{\sharp}\left(\frac{1}{N!}\sum_{\sigma\in\Sym_N}\gamma_\a(z_{\sigma(1)}-\rae,\dots,z_{\sigma(N_\a)}-\rae)\gamma_\b(z_{\sigma(N_{\a}+1))}-\rbe,\dots,z_{\sigma(N_{\a}+N_{\b})}-\rbe)\right), \\
&= \dfrac{N_\a}{N_{\a}+N_{\b}}\ra(z_i-\rae) + \dfrac{N_\b}{N_{\a}+N_{\b}}\rb(z_i-\rbe) = \rho_{\eta}(z_i),
\end{align*}
where we denoted by $\Sym_N$ the set of permutations of $N$ elements.
\end{proof}
Note that $\Pi(\re)$ is strictly larger than $\Pi_{\na,\nb}^\eta(\ra,\rb)$. Nevertheless, for $\eta\to 0$ one may expect that optimizers of $\VSCE[\re]$ will tend to a minimizer of the restriction of the functional from $\VSCE$ to $\Pi_{\na,\nb}^\eta(\ra,\rb)$. This is because we do not allow electron tunneling over large distances of $1/\eta$, and thus the requirement of $\na,\nb$ electrons in the molecules becomes sharp in the limit. Note that in \cite{bouchitte2021dissociating, de2021relaxed} the tunneling question is left open. The question of proving the absence of tunneling rigorously is a technical one, and is not the focus of this paper. Instead, here, we use plans that respect the molecular structure and consider their asymptotic.

\medskip
We note that the energy $\Vee$ is symmetric under permutation of the $N$ coordinates of $\R^{dN}=(\R^d)^N$, thus we have
\[
    \int_{\R^{dN}}\Vee(\vec z)d(\gamma)_{sym}(z)=\int_{\R^{dN}}\Vee(\vec z) d\gamma(\vec z),
\]
therefore the problem we study can be reformulated as follows:
\begin{eqnarray}\label{eq:dispersion}
    \lefteqn{\min\left\{\int_{\R^{dN}}\Vee(\vec z)d\gamma(\vec z) :\ \gamma\in \Pi_{\na,\nb}^\eta(\ra,\rb)\right\}}\nonumber\\
    &=&\min\left\{\int_{\R^{dN}}\Vee(\vec z)d\left((\tau_{\rae})_\sharp \gamma_\a \otimes (\tau_{\rbe})_\sharp \gamma_\b\right)(\vec z) :\  \gamma_\a\in \Pi_\na(\ra),\gamma_\b \in \Pi_\nb(\rb)\right\}\nonumber\\
    &=&\min\left\{\int_{\R^{d\na}}\int_{\R^{d\nb}}\Vee^\eta(\vec x,\vec y)d\gamma_\a(\vec x)d\gamma_\b(\vec y):\ \gamma_\a\in \Pi_\na(\ra),\gamma_\b \in \Pi_\nb(\rb)\right\}\nonumber\\
    &=&\min_{\gamma\in\Pi_\na(\ra)\otimes\Pi_\nb(\rb)}\int_{\R^{dN}}\Vee^\eta(\vec x, \vec y)d\gamma(\vec x,\vec y).
\end{eqnarray}
To justify the above, in the first step we used the symmetry of $\Vee$ and in the second step we used the rewriting \eqref{veta} of $\Vee$ in center-of-molecule coordinates and the structure of the plans which respect the coordinate splitting from $(\vec z)$ to $(\vec x, \vec y)$.

\noindent
\textbf{Taylor expansion of the Coulomb interaction energy when $\eta\to 0^+$.} We now compute the asymptotic expansion of the Coulomb electronic-electronic interaction in \eqref{veta} when the distance $R=\eta^{-1}$ of the cluster of molecules $\alpha\beta$ goes to $+\infty$.

\begin{prop}\label{prop:tayloreta}
Let $\eta>0$ be a positive number, $N_A,N_B$ be positive integers, $\ra,\rb\in\mathcal{P}(\mathbb R^d)$ and $V^\eta(\vec x, \vec y)$ as in \eqref{veta}. Then, for $\gamma \in \Pi_\na(\ra)\otimes\Pi_\nb(\rb)$ we have
\begin{align}\label{eq:taylorveta}
\int_{\R^{dN}}V^{\eta}(\vec{x},\vec{y})d\gamma &= \int_{\R^{dN}}V^{\alpha}_{ee}(\vec{x})d\gamma_{\alpha}+\int_{\R^{dN}}V_{ee}^{\beta}(\vec{x})d\gamma_{\beta} + U^{\eta}_{\mathrm{int}}[\ra,\rb] + \\
&+\dfrac{1}{2}\eta^3\int_{\R^{dN}}\sum^{\na}_{i=1}\sum^{\nb}_{j=1}\left(3(x_i^{1}-y_j^{1})^2-\vert x_i-y_j\vert^2\right)d\gamma + O(\eta^4),
\end{align}
where $\gamma_\a\in \Pi_\na(\ra),\gamma_\b \in \Pi_\nb(\rb)$ and the internal energy $U^{\eta}_{\mathrm{int}}[\ra,\rb]$ is defined by
\begin{equation}\label{eq:uint}
    U^{\eta}_{\mathrm{int}}[\ra,\rb] := \na\nb\left(\eta + \left(\int_{\R^d}x^1d\ra(x) - \int_{\R^d}y^1d\rb(y)\right)\eta^2\right).
\end{equation}
\end{prop}

\begin{proof}
We first compute the Taylor expansion of $V^{\eta}(\vec{x},\vec{y})$ at $\eta=0$:

\begin{align*}
    V^{\eta}(\vec{x},\vec{y}) &= V^{0}(\vec{x},\vec{y}) + \dfrac{dV}{d\eta}^{0}(\vec{x},\vec{y})\eta + \dfrac{1}{2}\dfrac{d^2V}{d\eta^2}^{0}(\vec{x},\vec{y})\eta^2 + O(\eta^3) \\
    &= \Vee^{\alpha}(\vec{x}) + \Vee^{\beta}(\vec{y}) + \na\nb\eta + \eta^2\sum^{\na}_{i=1}\sum^{\nb}_{j=1}(x_i^1-y_j^1) +\\
    &\qquad \qquad \qquad \qquad + \dfrac{1}{2}\eta^3\sum^{\na}_{i=1}\sum^{\nb}_{j=1}\left(3(x_i^1-y_j^1)^2-\vert x_i-y_j\vert^2\right) + O(\eta^3) 
\end{align*}
Let $\gamma\in\Pi_\na(\ra)\otimes\Pi_\nb(\rb)$. Then there exists  $\gamma_\a\in \Pi_\na(\ra)$ and $\gamma_\b \in \Pi_\nb(\rb)$ such that $\gamma=\left((\tau_{\rae})_\sharp \gamma_\a \otimes (\tau_{\rbe})_\sharp \gamma_\b\right)_{sym}$. Notice that, upon integrating the above expression against $\gamma$, the terms in the development of order $1$ and $\eta^2$ are given by
\[
\int_{\R^{dN}}V^{0}(\vec{x},\vec{y})d\gamma = \int_{\R^{dN}}V^{\alpha}_{ee}(\vec{x})d\gamma_{\alpha}(\vec{x})+\int_{\R^{dN}}V_{ee}^{\beta}(\vec{x})d\gamma_{\beta}(\vec{y}), 
\]
and
\[
\int_{\R^{dN}}\sum^{\na}_{i=1}\sum^{\nb}_{j=1}(x_i^1-y_j^1)d\gamma =\na\nb\left(\int_{\R^d}x^1d\ra(x) - \int_{\R^d}y^1d\rb(y)\right).
\]
Therefore, we directly obtain \eqref{eq:taylorveta} and \eqref{eq:uint}.
\end{proof}

\noindent
\section{Multi-marginal Optimal Transport problem}\label{sec:MOT}

In this section, we will present partial results on the existence and uniqueness of Monge solutions for the multi-marginal problem with cost function arising in the term of order $\eta^{3}$ in the expansion of $\eta\to 0^+$ in \eqref{eq:taylorveta}, namely
\begin{equation}\label{eq:MOTdis}
\min_{\gamma\in\Pi(\ra)\otimes\Pi(\rb)}\frac{1}{2}\int_{\R^{dN}}\sum^{\na}_{i=1}\sum^{\nb}_{j=1}\left(3(x_i^1-y_j^1)^2-\vert x_i-y_j\vert^2\right)d\gamma(x_1,\dots,x_\na,y_{1},\dots,y_\nb).
\end{equation}

The first theorem shows that if the marginal distributions $\ra$ and $\rb$ have finite second moments, then the problem \eqref{eq:MOTdis} is equivalent to the multi-marginal optimal transport problem introduced in \eqref{maincost}.

\begin{prop}
Let $\na,\nb \in \N$ be integers, $N=\na+\nb$, $X_i \subseteq \R^{d}$, $i=1, \ldots, N$ be open sets and $X:=\prod_{i=1}^{N}X_i$. If $\ra$ and $\rb$ are probability measures in $\R^d$ having finite moments, then, for any transport plan $\gamma\in\Pi_{\na}(\ra)\otimes\Pi_{\nb}(\rb)$, we have
\[
\frac{1}{2}\int_{X}\sum^{\na}_{i=1}\sum^{\nb}_{j=1}\left(3(x_i^1-y_j^1)^2-\vert x_i-y_j\vert^2\right)d\gamma = \int_{X}\sum^{\na}_{i=1}\sum^{\nb}_{j=1}x_i\cdot y_j^{*}d\gamma(x_1,\dots,x_\na,y_{1},\dots,y_\nb) + C.
\]
where $y_j^{*}=(-2y_j^{1}, y_j^{2}, \ldots, y_j^{d})$ for each $j\in \lbrace 1,\dots,\nb\rbrace$ and $C\in\R$ is a constant. In particular,
\[
\underset{\gamma\in\Pi_{\na}(\ra)\otimes\Pi_{\nb}(\rb)}{\mathsf{argmin}}\int_{X}\sum^{\na}_{i=1}\sum^{\nb}_{j=1}\left(3(x_i^1-y_j^1)^2-\vert x_i-y_j\vert^2\right)d\gamma = \underset{\gamma\in\Pi_{\na}(\ra)\otimes\Pi_{\nb}(\rb)}{\mathsf{argmin}}\int_{X}\sum^{\na}_{i=1}\sum^{\nb}_{j=1}x_i\cdot y_j^{*}d\gamma.
\]
\end{prop}

\begin{proof}

We just need to expand the squares in the cost function in \eqref{eq:MOTdis} and notice that for all $\gamma\in\Pi_{\na}(\ra)\otimes\Pi_{\nb}(\rb)$ we have
\begin{align*}
\frac{1}{2}\int_{X}\sum^{\na}_{i=1}\sum^{\nb}_{j=1}\left(3(x_i^1-y_j^1)^2-\vert x_i-y_j\vert^2\right)d\gamma &= \int_{X}- \sum^{\na}_{i=1}\sum^{\nb}_{j=1}3x_i^1y_j^1 + \sum^{\na}_{i=1}\sum^{\nb}_{j=1}x_i\cdot y_jd\gamma + C \\
&=\int_{X} \sum^{\na}_{i=1}\sum^{\nb}_{j=1}x_i\cdot y^*_jd\gamma + C,
\end{align*}
where $y_j^{*}=(-2y_j^{1}, y_j^{2}, \ldots, y_j^{d})$ for each $j\in \lbrace 1,\dots,\nb\rbrace$ and the constant $C$ is given by
\begin{align*}
C &= \int_{\R^{dN}} \sum^{\na}_{i=1}\sum^{\nb}_{j=1}3(x_i^1)^2 + \sum^{\na}_{i=1}\sum^{\nb}_{j=1}3(y_j^1)^2 - \sum^{\na}_{i=1}\sum^{\nb}_{j=1}x_i^2 - \sum^{\na}_{i=1}\sum^{\nb}_{j=1}y_j^2d\gamma \\
&=\na\nb\int_{\R^{d}} 3(x_i^1)^2d\ra+\na\nb\int_{\R^{d}} 3(y_j^1)^2d\rb - \na\nb\int_{\R^{d}} (x_i)^2d\ra-\na\nb\int_{\R^{d}} (y_j)^2d\rb.
\end{align*}
\end{proof}

Assuming finite second moments for the measures $\ra$ and $\rb $, the corresponding Kantorovich dual problem of \eqref{eq:MOTdis} is the variational problem below. Notice that the existence of the maximizer is guaranteed in G. H. Kellerer \cite{Kel-ZWG-84}.
\begin{equation*}
\max_{u\in L^1(\ra\otimes\rb)} \left\lbrace \na\int_{\R^{d\na}} u(x)d\ra(x) + \nb\int_{\R^{d\nb}} u(y)d\rb(y) \, : \,\sum^{\na}_{i=1}u(x_i)+\sum^{\nb}_{j=1}u(y_j)\leq \sum^{\na}_{i=1}\sum^{\nb}_{j=1}x_i\cdot y^*_j  \right\rbrace.
\end{equation*}

\subsection{Existence and uniqueness of Monge solutions}

 Let us now recall some main concepts from \cite{KimPass2014,PassVar2022}.
 \begin{deff}\label{splittingsets}
 Let $X_i \subseteq \R^{d}$, $i=1, \ldots, N$ be open sets, $X:=\prod_{i=1}^{N}X_i$ and $c:X\to\R$ be a cost function. A set $S\subseteq X$ is called a $c$-splitting set if there are Borel functions $u_i:X_i \mapsto \mathbb{R}$ such that 
\begin{equation}\label{Art3:6}
\sum_{i=1}^{N}u_i(x_{i})\leq c(x_1, \ldots,x_N)
\end{equation}
for every $(x_1, \ldots, x_N)\in X$, and whenever $(x_1, \ldots, x_N)\in S$ equality holds. The functions $u_1(x_1), \ldots, u_N(x_N)$ are called $c$-splitting functions for $S$.  
 \end{deff}
Assume $\{k_{i}\}_{i=1}^{r}\subseteq \{2, \ldots, N\}$, with $k_{1}< k_{2}<\ldots < k_{r}$. For a given $N$-tuple of Borel functions $(u_{1},\ldots, u_{N})$ satisfying inequality \eqref{Art3:6}, and $x_1^{0}\in X_1$, we define 

\begin{align*}
M_{x_1^{0}k_1\ldots k_r}(u_1, \ldots, u_N)&:=\Big\{ (x_{2}, \ldots, x_{N})\in \prod_{i=2}^{N}X_i: \text{$Du_{k_{i}}(x_{k_{i}})$ exists for each }\\
&\qquad \qquad\quad i=1,\ldots,r\;\;\text{and}\;u_1(x_1^{0})+\sum_{i=2}^{N}u_i(x_i)=c(x_1^{0}, x_2, \ldots,x_N)\Big\}.
\end{align*}

\begin{rmk}\label{Eqn:16}
Note that for a given set of the form $M_{x_1^{0}k_1\ldots k_r}(u_1, \ldots, u_N)$, if $Du_1(x_1^{0})$ exists then $D_{x_1}c(x_1^{0}, x_2, \ldots, x_N)$ exists for every $(x_2, \ldots, x_N)\in M_{x_1^{0}k_1\ldots k_r}(u_1, \ldots, u_N) $ and 
$$Du_1(x_1^{0})=D_{x_1}c(x_1^{0}, x_2, \ldots, x_N)$$ (for details see Lemma 2.2 in \cite{PassVar2022}). By definition of the set $M_{x_1^{0}k_1\ldots k_r}(u_1, \ldots, u_N)$ we also have $Du_{k_{i}}(x_{k_{i}})=D_{x_{k_{i}}}c(x_1^{0}, x_2, \ldots, x_N)$ for every $i=1, \ldots, r$.
\end{rmk}
\begin{deff}\label{Art3:10}
Let $X_i \subseteq \R^{d}$, $i=1, \ldots, N$ be open sets, $X:=\prod_{i=1}^{N}X_i$ and $c:X\to\R$ be a semi-concave cost function. Let $\{k_{i}\}_{i=1}^{r}\subseteq \{2, \ldots, N\}$, with $k_{1}< k_{2}<\ldots < k_{r}$. We say $c$ is twisted on $c$-splitting sets  with respect to the variables $x_{1}, x_{k_{1}},\ldots, x_{k_{r}}$  if 
for every $N$-tuple of Borel functions $(u_1,\ldots, u_N)$ satisfying inequality \eqref{Art3:6} and for every $x_1^{0}\in X_1$ with $M_{x_1^{0}k_1\ldots k_r}\neq \emptyset$, we get that the map 
$$(x_{2}, \ldots, x_{N})\mapsto D_{x_{1}}c(x_{1}^{0},x_2,\ldots, x_{N}) $$ is injective on
the subset of  $M_{x_1^{0}k_1\ldots k_r}$ where $D_{x_1}c(x_1^{0}, x_2, \ldots, x_N)$ exists. 
\end{deff}
Note that the special case of $c$ being twisted on $c$-splitting sets with respect to the variable $x_1$ is equivalent to  the twisted on $c$-splitting sets condition.
\begin{rmk}
The main result in \cite{PassVar2022} states that if $c$ is twisted on $c$-splitting sets  with respect to the variables $x_{1}, x_{k_{1}},\ldots, x_{k_{r}}$, with  $\{k_{i}\}_{i=1}^{r}\subseteq \{2, \ldots, N\}$, $k_{1}< k_{2}<\ldots < k_{r}$, then the solution $\gamma$ in (\ref{KP}) is concentrated on a graph of a measurable map $u:\mathbb R^d\to \mathbb R^{rd}$ and it is unique, as long as $\rho_{1},\rho_{k_1}, \ldots,\rho_{k_r}$ are absolutely continuous with respect to the $d$-dimensional Lebesgue measure $\mathcal{L}^{d}$.\par

In this work, to get uniqueness we need the standard regularity condition on $\rho_1$. Under this assumption we can focus on the set formed of all $x_1^{0}\in X_1$ such that $Du_1(x_1^{0})$ exists for some $N$-tuple  $(u_1, \ldots, u_N)$ of Borel functions satisfying inequality \eqref{Art3:6}, and so if $M_{x_1^{0}k_1\ldots k_r}\neq \emptyset$, the map $$(x_{2}, \ldots, x_{N})\mapsto D_{x_{1}}c(x_{1}^{0},x_2,\ldots, x_{N}) $$ is injective on
the subset of  $M_{x_1^{0}k_1\ldots k_r}$ where $D_{x_1}c(x_1^{0}, x_2, \ldots, x_N)$ exists if and only if $M_{x_1^{0}k_1\ldots k_r}$ is a singleton. See details in \cite{PassVar2022}. 
\end{rmk}
\par

Our main result establishes that the cost in \eqref{maincost} provides unique Monge solutions under two regularity conditions. Note that in the  case $N_\alpha=1$ the cost function in \ref{maincost} reduces to the simple structure $x_1\cdot \sum_{j=1}^{N_{\beta}}y_j^{*}$, which clearly gives uniqueness if the first marginal $\rho_1$ is absolutely continuous with respect to the $d$-dimensional Lebesgue measure $\mathcal{L}^{d}$. In what follows we prove uniqueness for the case $N_{\alpha}>1$.

\begin{teo}\label{Eqn:10}
Let $\rho_i$ be Borel probability measures on open bounded sets $X_{i}\subseteq \mathbb{R}^{d}$, $i=1,\ldots, N$, with $\rho_1$ absolutely continuous with respect to the $d$-dimensional Lebesgue measure $\mathcal{L}^{d}$. Assume $N_{\alpha}>1$ and there exists $p\in \{1, \ldots, N_\beta\}$ such that $\rho_{N_{\alpha} +p}$ is absolutely continuous with respect to the Lebesgue measure $\mathcal{L}^{d}$.
Then the multi-marginal optimal transport problem \eqref{maincost} 
\[
\inf_{\gamma\in\Pi(\ra)\otimes\Pi(\rb)}\int_{\R^{dN}}c(\vec{x},\vec{y})d\gamma := \inf_{\gamma\in\Pi(\ra)\otimes\Pi(\rb)}\int_{\R^{d(\na+\nb)}}\sum^{\na}_{i=1}\sum^{\nb}_{j=1}x_i\cdot y_j^{*}d\gamma(x_1,\dots,x_\na,y_{1},\dots,y_\nb),
\]
admits a unique Monge solution. As before, $y_j^{*}=(-2y_j^{1}, y_j^{2}, \ldots, y_j^{d})$ for each $j\in \lbrace 1,\dots,\nb\rbrace$.
\end{teo}
\begin{proof}
Let $(u_1,\ldots, u_N)$ be an $N$-tuple of Borel functions satisfying inequality \eqref{Art3:6} and fix 
 $x_1^{0}\in X_1$ such that $Du_1(x_1^{0})$ exists and $ M_{x_1^{0}(N_{\alpha}+p)}(u_1,\ldots, u_N) \neq \emptyset$. We want to prove that $M_{x_1^{0}(N_{\alpha}+p)}(u_1,\ldots, u_N)$ is a singleton. 
    \par 
    Let $(x_{2k},\ldots,x_{N_{\alpha}k}, y_{1k}, \ldots, y_{N_{\beta}k})\in M_{x_1^{0}(N_{\alpha}+p)}$, $k=1,2$. Then $c$ is differentiable with respect to $x_{1}$ at $(x_{1}^{0},x_{2k},\ldots,x_{N_{\alpha}k}, y_{1k}, \ldots, y_{N_{\beta}k})$ (see Remark \ref{Eqn:16}) and it satisfies
    \begin{equation}\label{Theorem 1.2-1}
        Du_{1}(x_{1}^{0})= D_{x_{1}}c(x_{1}^{0},x_{2k},\ldots,x_{N_{\alpha}k}, y_{1k}, \ldots, y_{N_{\beta}k})=\sum_{j=1}^{N_{\beta}}y_{jk}^{*}, \quad k=1,2.
    \end{equation}
    We will prove that $x_{i1}=x_{i2}$ and $y_{j1}=y_{j2}$ for each $i\in \{2, \ldots,N_{\alpha}\}$ and $j\in \{1, \ldots,N_{\beta}\}$.
    \par
    From equation \eqref{Theorem 1.2-1} we get 
    \begin{equation}\label{Eqn:1}
        \sum_{j=1}^{N_{\beta}}y_{j1}^{*}=\sum_{j=1}^{N_{\beta}}y_{j2}^{*}.
    \end{equation}
    Fix $s\in A:=\{1, \ldots, N_{\beta}\}$ and $k\in \{1,2\}$, and for convenience of notation set $x_1^{0}=x_{11}=x_{12}$. Note that
    \begin{align*}
        \left\lbrace y_{jk}\right\rbrace_{j\in A\setminus \{s\}}&\in\mathsf{argmin}\Bigg\{ \left\lbrace y_{j}\right\rbrace_{j\in A\setminus \{s\}} \mapsto c(x_{1}^{0},x_{2k}, \ldots,x_{N_{\alpha}k}, y_{1}, \ldots, y_{s-1}, y_{sk}, y_{s+1}, \ldots, y_{N_{\beta}})\\
        & \qquad\qquad \qquad-\sum_{i=1}^{N_{\alpha}}u_{i}(x_{ik}) - \sum_{\underset{j\neq s}{j=1}}^{N_{\beta}}u_{N_\alpha + j}(y_j)-u_{N_\alpha + s}(y_{sk})\Bigg\}\\
        &=\mathsf{argmin}\Bigg\{ \left\lbrace y_{j}\right\rbrace_{j\in A\setminus \{s\}} \mapsto c(x_{1}^{0},x_{2k}, \ldots,x_{N_{\alpha}k}, y_{1}, \ldots, y_{s-1}, y_{sk}, y_{s+1}, \ldots, y_{N_{\beta}})\\
    & \qquad\qquad \qquad-\sum_{\underset{j\neq s}{j=1}}^{N_{\beta}}u_{N_\alpha + j}(y_j)\Bigg\}\\
    & =\mathsf{argmin}\Bigg\{ \left\lbrace y_{j}\right\rbrace_{j\in A\setminus \{s\}} \mapsto  \sum_{i=1}^{N_{\alpha}}\sum_{\underset{j\neq s}{j=1}}^{N_{\beta}}x_{ik}\cdot y_{j}^{*}-\sum_{\underset{j\neq s}{j=1}}^{N_{\beta}}u_{N_\alpha + j}(y_j)\Bigg\}
    \end{align*}
    Then 
    \begin{equation}\label{Eqn:2}
        \sum_{i=1}^{N_{\alpha}}\sum_{\underset{j\neq s}{j=1}}^{N_{\beta}}x_{i2}\cdot y_{j2}^{*}-\sum_{\underset{j\neq s}{j=1}}^{N_{\beta}}u_{N_\alpha + j}(y_{j2})\leq \sum_{i=1}^{N_{\alpha}}\sum_{\underset{j\neq s}{j=1}}^{N_{\beta}}x_{i2}\cdot y_{j1}^{*}-\sum_{\underset{j\neq s}{j=1}}^{N_{\beta}}u_{N_\alpha + j}(y_{j1}),
    \end{equation}
    \[
        \sum_{i=1}^{N_{\alpha}}\sum_{\underset{j\neq s}{j=1}}^{N_{\beta}}x_{i1}\cdot y_{j1}^{*}-\sum_{\underset{j\neq s}{j=1}}^{N_{\beta}}u_{N_\alpha + j}(y_{j1})\leq \sum_{i=1}^{N_{\alpha}}\sum_{\underset{j\neq s}{j=1}}^{N_{\beta}}x_{i1}\cdot y_{j2}^{*}-\sum_{\underset{j\neq s}{j=1}}^{N_{\beta}}u_{N_\alpha + j}(y_{j2}).
    \]
    Adding the above inequalities and eliminating similar terms we get
    \[
        \sum_{i=1}^{N_{\alpha}}\sum_{\underset{j\neq s}{j=1}}^{N_{\beta}}x_{i2}\cdot y_{j2}^{*}+\sum_{i=1}^{N_{\alpha}}\sum_{\underset{j\neq s}{j=1}}^{N_{\beta}}x_{i1}\cdot y_{j1}^{*}\leq \sum_{i=1}^{N_{\alpha}}\sum_{\underset{j\neq s}{j=1}}^{N_{\beta}}x_{i2}\cdot y_{j1}^{*}+ \sum_{i=1}^{N_{\alpha}}\sum_{\underset{j\neq s}{j=1}}^{N_{\beta}}x_{i1}\cdot y_{j2}^{*},
    \]
    or equivalently,
    \begin{equation}\label{Eqn:12}
     \sum_{i=1}^{N_{\alpha}}\left(x_{i2}-x_{i1}\right)\cdot \sum_{\underset{j\neq s}{j=1}}^{N_{\beta}}\left(y_{j2}^{*}-y_{j1}^{*}\right)\leq 0.   
    \end{equation}
    On the other hand we have
    \begin{align*}
         y_{sk}&\in\mathsf{argmin}\Bigg\{ y_{s} \mapsto c(x_{1}^{0},x_{2k}, \ldots,x_{N_{\alpha}k}, y_{1k}, \ldots, y_{(s-1)k}, y_{s}, y_{(s+1)k}, \ldots, y_{N_{\beta}k})-u_{N_\alpha + s}(y_s)\Bigg\},
        \end{align*}
        hence, we also get
        \begin{equation}\label{Eqn:13}
         \sum_{i=1}^{N_{\alpha}}\left(x_{i2}-x_{i1}\right)\cdot \left(y_{s2}^{*}-y_{s1}^{*}\right)\leq 0.   
        \end{equation}   
    Adding inequalities \eqref{Eqn:12} and \eqref{Eqn:13} and using \eqref{Eqn:1} we conclude that in particular, equality holds in \eqref{Eqn:2}. Then
    \begin{align*}
      \sum_{i=1}^{N_{\alpha}}u_{i}(x_{i2}) + u_{N_\alpha + s}(y_{s2})-\sum_{i=1}^{N_{\alpha}}x_{i2}\cdot y_{s2}^{*}&=  \sum_{i=1}^{N_{\alpha}}\sum_{\underset{j\neq s}{j=1}}^{N_{\beta}}x_{i2}\cdot y_{j2}^{*}-\sum_{\underset{j\neq s}{j=1}}^{N_{\beta}}u_{N_\alpha + j}(y_{j2})\\
      &= \sum_{i=1}^{N_{\alpha}}\sum_{\underset{j\neq s}{j=1}}^{N_{\beta}}x_{i2}\cdot y_{j1}^{*}-\sum_{\underset{j\neq s}{j=1}}^{N_{\beta}}u_{N_\alpha + j}(y_{j1}).
    \end{align*}
    $(x_{22}, \ldots,x_{N_{\alpha}2}, y_{11}, \ldots, y_{(s-1)1}, y_{s2}, y_{(s+1)1}, \ldots, y_{N_{\beta}1})\in M_{x_1^{0}(N_{\alpha}+p)}(u_1,\ldots, u_N)$. Then by Remark \ref{Eqn:16} we get
    \begin{align*}
        \sum_{\underset{j\neq s}{j=1}}^{N_{\beta}}y_{j1}^{*} + y_{s2}^{*}&= D_{x_{1}}c(x_1^{0}, x_{22}, \ldots,x_{N_{\alpha}2}, y_{11}, \ldots, y_{(s-1)1}, y_{s2}, y_{(s+1)1}, \ldots, y_{N_{\beta}1})\\
        &=Du_{1}(x_{1}^{0})\\
        & =D_{x_{1}}c(x_{1}^{0},x_{21},\ldots,x_{N_{\alpha}1}, y_{11}, \ldots, y_{N_{\beta}1})\\
        &=\sum_{j=1}^{N_{\beta}}y_{j1}^{*},
    \end{align*}
    which implies $y_{s2}^{*}=y_{s1}^{*}$; that is,
    \begin{equation}\label{Eqn:4}
        y_{s2}=y_{s1} \quad \text{for every $s\in \{1, \ldots,N_{\beta}\}.$}
    \end{equation}

    In particular we have $y_{p1}=y_{p2}$ and by Remark \ref{Eqn:16}
    \begin{equation*}
        Du_{N_{\alpha} +p}(y_{p1})=Du_{N_{\alpha} +p}(y_{p2})= D_{y_{p}}c(x_{1}^{0},x_{2k},\ldots,x_{N_{\alpha}k}, y_{1k}, \ldots, y_{N_{\beta}k})=\sum_{i=1}^{N_{\alpha}}x_{ik}^{*},
    \end{equation*}
    where $x_{ik}^{*}=(-2x_{ik}^{1}, x_{ik}^{2}, \ldots, x_{ik}^{d})$ for each $i$, $k=1,2$.\par

    Using the above equation, we proceed to make a straightforward adaptation of the arguments
    used from Equation \eqref{Theorem 1.2-1} to Equation \eqref{Eqn:4} to prove that $x_{s1}=x_{s2}$ for every $s\in \{2, \ldots, N_{\alpha} \}$. This completes the proof of the theorem.
\end{proof}
Note that cost function in  \eqref{maincost} does not include interaction among the variables $x_1, \ldots, x_{N_\alpha}$, $N_{\alpha}\geq 2$. However, all of them interact with $y_j$ for any  $j \in \{1, \ldots \beta \}$.  So, it is natural to connect the variables  $x_1, \ldots, x_{N_\alpha}$ via one of the $y_j$, using  an extra regularity condition on $\rho_{N_{\alpha}+j}$. The following simple Lemma shows that this condition is necessary.
\begin{lemma}\label{Eqn:15}
Consider Dirac measures $\rho_{N_{\alpha}+j}= \delta_{\hat{y}_{j}}$, $j=1, \ldots, N_{\beta}$, and let $\rho_i$ be any non Dirac measure, for some $i\in \{2, \ldots, N_{\alpha}\}$. Then there exists a solution of non-Monge form to the Kantorovich problem.
\end{lemma}
\begin{proof}
    For any transport plan $\gamma$ to the Kantorovich problem we have $\gamma=\overline{\gamma}\otimes \delta_{\hat{y}_{1}} \otimes \ldots \otimes\delta_{\hat{y}_{N_{\beta}}}$, where $\overline{\gamma}$ is a probability measure on $\prod_{i=1}^{N_{\alpha}}X_i$ with marginals $\rho_1, \ldots, \rho_{N_{\alpha}}$. Then
    \begin{align*}
        \displaystyle \int_{\prod_{i=1}^{N}X_i} cd\gamma & = \displaystyle \int_{\prod_{i=1}^{N_{\alpha}}X_i} c(x_1,\ldots, x_{N_{\alpha}}, \hat{y}_{1}, \ldots, \hat{y}_{N_{\beta}})   d\overline{\gamma}\\
        &= \sum_{i=1}^{N_{\alpha}} \sum_{j=1}^{N_{\beta}} \displaystyle \int_{\prod_{i=1}^{N_{\alpha}}X_i}x_i\cdot (\hat{y_j})^{*}d\overline{\gamma}\\
        &= \sum_{i=1}^{N_{\alpha}} \sum_{j=1}^{N_{\beta}} \displaystyle \int_{X_i}x_i\cdot (\hat{y_j})^{*}d\rho_{i}.
    \end{align*}
    We then deduce that any transport plan $\gamma$ is a solution. In particular, $\rho_1 \otimes \ldots \otimes \rho_N$ is solution of non-Monge form, as there exists $i\in \{2, \ldots, \beta \}$ with $\rho_i$ non Dirac measure.
\end{proof}
\section*{Acknowledgment}
AG and AVJ acknowledge support of our research by the Canada Research Chairs Program and Natural Sciences and Engineering Research Council of Canada. MP acknowledges support from CenIA (Centro Nacional de Inteligencia Artificial) and was funded by Chilean Fondecyt Regular grant n.1210462 entitled ``Rigidity, stability and uniformity for large point configurations'' and from the Chilean Centro Nacional de Inteligencia Artificial.  

\bibliographystyle{siam}

\end{document}